\documentclass[12pt]{amsart}
\usepackage[cmtip,all]{xy}
\usepackage{graphicx}
\usepackage{amssymb}
\usepackage{mathrsfs}
\usepackage{amsfonts}
\usepackage{amsmath}

\usepackage[pagebackref]{hyperref}

\newtheorem{theorem}{Theorem}[section]
\newtheorem{lemma}[theorem]{Lemma}

\theoremstyle{definition}

 \theoremstyle{remark}
\newtheorem{remark}[theorem]{Remark}
\newtheorem{corollary}[theorem]{Corollary}

 \numberwithin{equation}{section}

\begin{document}

\title[Trudinger-Moser   inequalities for  Aharonov-Bohm Magnetic    fields]{Trudinger-Moser and Hardy-Trudinger-Moser inequalities for the Aharonov-Bohm Magnetic    field}

%    Information for first author
\author{Guozhen Lu}
%    Address of record for the research reported here
\address{Department of Mathematics,   University of Connecticut, Storrs, CT 06290, USA}
%    Current address

\email{guozhen.lu@uconn.edu}

\author{Qiaohua  Yang}
%    Address of record for the research reported here
\address{School of Mathematics and Statistics, Wuhan University, Wuhan, 430072, People's Republic of China}
%    Current address

\email{qhyang.math@whu.edu.cn}

\thanks{ The first author was partly supported by a Simons collaboration grant and a Simons Fellowship from the Simons Foundation. The second author  was partly  supported by   the National Natural
Science Foundation of China (No.12071353).  }

 %\thanks{Support information for the second author.}

%    General info

%\date{January 1, 2001 and, in revised form, June 22, 2001.}

%\dedicatory{This paper is dedicated to our advisors.}


\begin{abstract}
The main results of this paper concern sharp constant of  the Trudinger-Moser   inequality in $\mathbb{R}^{2}$  for  Aharonov-Bohm  magnetic fields. This is a borderline case of the Hardy type inequalities  for  Aharonov-Bohm  magnetic fields in $\mathbb{R}^2$  studied by A. Laptev and T. Weidl.  As an application, we obtain the exact asymptotic estimates on best constants of magnetic Hardy-Sobolev inequalities.
In order to achieve our goal, we introduce a new operator $T_{a}$ on the unit circle $\mathbb{S}^{1}$ and give   the   asymptotic estimates of the heat kernel $e^{tT_{a}}$ via
the Poisson summation formula. Finally,  we  show that such  Trudinger-Moser inequalities in the unit ball $\mathbb{B}^{2}$
can be improved   via  subtraction of an additional
Hardy term to derive a Hardy-Trudinger-Moser inequality.
\end{abstract}

\maketitle

%\section*{This is an unnumbered first-level section head}

\section{Introduction}
The  Hardy inequality reads that for $n\geq3$ and for each complex-valued function
$u\in C^{\infty}_{0}(\mathbb{R}^{n})$,
\begin{equation}\label{1.1a}
\int_{\mathbb{R}^{n}}|\nabla u|^{2}dx\geq
\left(\frac{n-2}{2}\right)^{2}\int_{\mathbb{R}^{n}}\frac{|u|^{2}}{|x|^{2}}dx
\end{equation}
and $\left(\frac{n-2}{2}\right)^{2}$ is sharp.

When $n=2$, the  above Hardy's inequality fails with any positive constant replacing $\left(\frac{n-2}{2}\right)^{2}$, which is zero.
However, we can still prove
some nontrivial Hardy inequalities if the singular weight $\frac{1}{\left\vert
x\right\vert ^{2}}$ is weakened by adding a logarithmic term and/or including some additional assumptions on the functions $u$.
Namely,

\[
\int\limits_{\mathbb{R}^{2}}\left\vert \nabla u\right\vert ^{2}dx\geq
C\int\limits_{\mathbb{R}^{2}}\frac{\left\vert u\right\vert ^{2}}{\left\vert
x\right\vert ^{2}\left(  1+\left\vert \ln x\right\vert ^{2}\right)  }dx\text{
if }\int\limits_{\left\vert x\right\vert =1}u\left(  x\right)  dx=0
\]
or%
\[
\int\limits_{\mathbb{R}^{2}}\left\vert \nabla u\right\vert ^{2}dx\geq
C\int\limits_{\mathbb{R}^{2}}\frac{\left\vert u\right\vert ^{2}}{\left\vert
x\right\vert ^{2}}dx\text{ if }\int\limits_{\left\vert x\right\vert
=r}u\left(  x\right)  dx=0\text{, }\forall r>0\text{.}%
\]
See, for instance, \cite{ET99, Sol94}.

However, for some magnetic forms, the Hardy
inequality in the form of \eqref{1.1a} with a positive constant becomes possible even in dimension two.
In fact, it has been shown by  Laptev and   Weidl (\cite{lw}) that
\begin{equation}\label{1.1}
\int_{\mathbb{R}^{2}}|(\nabla+i\mathbf{A}) u|^{2}dx\geq  \min_{n\in \mathbb{Z}}(n-a)^{2}\int_{\mathbb{R}^{2}}\frac{|u|^{2}}{|x|^{2}}dx,\;\;\;\;
u\in C^{\infty}_{0}(\mathbb{R}^{2}\setminus\{0\}),
\end{equation}
where $\mathbb{Z}$ is the set of integers and $\mathbf{A}$ is the  Aharonov-Bohm Magnetic potential defined by
\begin{align}\label{TheFieldA}
\mathbf{A}=\frac{a}{|x|^{2}}(x_{2},-x_{1}),\;\;a\in\mathbb{R}.
\end{align}
Furthermore, the constant $\min\limits_{n\in \mathbb{Z}}(n-a)^{2}$ is sharp. We refer to \cite{Aer14}, \cite{AFT03}, \cite{BEL15},  \cite{BLS04}, \cite{CK16}, \cite{CKLL},  \cite{FKLV20},  \cite{Kre19}, \cite{LamLu-VJM}, \cite{Sol94}  and references therein
for more work on  the $L^2$ and   $L^{p}$ $(p\neq 2)$ Hardy type inequalities with magnetic fields.

For simplicity, we set $re^{i\theta}=x_{1}+ix_{2}$ and
$\nabla_{\mathbf{A}}=\nabla+i\mathbf{A}$.
Then
\begin{align}\label{1.2}
|\nabla_{\mathbf{A}}u|^{2}=\left|\frac{\partial u}{\partial r}\right|^{2}+\frac{1}{r^{2}}
\left|\left(\frac{\partial }{\partial \theta}-ia\right)u\right|^{2}.
\end{align}
Denote by
\begin{align*}%\label{1.4}
\textrm{H}_{\mathbf{A}}(\mathbb{R}^{2})=\{\psi\in L^{2}(\mathbb{R}^{2}): \nabla_{\mathbf{A}}\psi\in L^{2}(\mathbb{R}^{2})\}.
\end{align*}
It has been shown by  Bonheure,   Dolbeault, Esteban,   Laptev and
Loss (\cite{bdel1}) that for $\lambda>-\min\limits_{n\in \mathbb{Z}}(n-a)^{2}$ and $p>2$,
there is an optimal function $\lambda\rightarrow\mu_{p}(\lambda)$ which is
monotone increasing and concave  such that the following magnetic Hardy-Sobolev inequality holds
\begin{align}\label{1.3}
\int_{\mathbb{R}^{2}}|\nabla_{\mathbf{A}}u|^{2}dx+\lambda
\int_{\mathbb{R}^{2}}\frac{|u|^{2}}{|x|^{2}}dx\geq \mu_{p}(\lambda)\left(\int_{\mathbb{R}^{2}}\frac{|u|^{p}}{|x|^{2}}dx\right)^{\frac{2}{p}}, \; u\in\textrm{H}_{\mathbf{A}}(\mathbb{R}^{2}).
\end{align}
Without loss of generality, one can  assume $0\leq a\leq \frac{1}{2}$ and thus  $\min\limits_{n\in \mathbb{Z}}(n-a)^{2}=a^{2}$. It has been shown in \cite{bdel1} that for
 $0<a<\frac{1}{2}$ and $-a^{2}<\lambda\leq\lambda_{\star}$, where $(\lambda_{\star}+a^{2})(p^{2}-4)=4(1-4a^{2})$, equality in (\ref{1.3}) is achieved by the function
\begin{align*}
u(x)=(|x|^{\alpha}+|x|^{-\alpha})^{-\frac{2}{p-2}},\; x\in\mathbb{R}^{2},\;\alpha=\frac{p-2}{2}\sqrt{\lambda+a^{2}}
\end{align*}
and the optimal constant is
\begin{align}\label{1.4}
\mu_{p}(\lambda)=\frac{p}{2}(2\pi)^{1-\frac{2}{p}}(\lambda+a^{2})^{1+\frac{2}{p}}\left(\frac{2\sqrt{\pi}
\Gamma(\frac{p}{p-2})}{(p-2)\Gamma(\frac{p}{p-2}+\frac{1}{2})}\right)^{1-\frac{2}{p}}.
\end{align}

It is not easy to see how the constant $\mu_{p}(\lambda)$ behave asymptotically as $p$ goes to infinity. We note that if $p$ goes to infinity in (\ref{1.3}),  one would expect that the appropriate function space for $u$ to belong to is the exponential class. In such a case,
it is natural to consider the Trudinger type inequality, which is known as the  borderline case of Sobolev inequality (see Trudinger \cite{t},  Yudovich \cite{Yu}, Pohozaev \cite{Po}).
In 1971, Moser  sharpened  the Trudinger inequality  in   \cite{mo} by finding the optimal constant.
\begin{theorem}[Moser]\label{theorem1.1} Let $\Omega$ be a domain with finite measure in Euclidean space $\mathbb{R}^n$, $n\ge2$. Then there exists a sharp constant $\beta_{n}=n\left(
\frac{n\pi^{\frac{n}{2}}}{\Gamma(\frac{n}{2}+1)}\right)  ^{\frac{1}{n-1}}%
$ such that
\begin{displaymath}
\frac{1}{|\Omega|}\int_{\Omega}\exp(\beta|u|^{\frac{n}{n-1}})dx\le c_0
\end{displaymath}
for any $\beta\le\beta_n$, any $u\in C^{\infty}_0(\Omega)$ with $\int_{\Omega}|\nabla u|^ndx\le1$. This constant $\beta_n$ is sharp in the sense that if $\beta>\beta_n$, then the above inequality can no longer hold with some $c_0$ independent of $u$.
\end{theorem}

We shall focus on the case $n=2$ only in this paper  and then the corresponding  optimal constant $\beta_{2}=4\pi$.

\medskip

One of the main results of this paper  is the following borderline case of the Hardy inequality \eqref{1.3} of   Bonheure,   Dolbeault, Esteban,   Laptev and
Loss (\cite{bdel1}), namely a Trudinger-Moser type inequality. Our result sharpened the integrability in \eqref{1.3}.

\begin{theorem} \label{th1.2} Let $\mathbf{A}$ be the  Aharonov-Bohm Magnetic potential defined in \eqref{TheFieldA} and
 $0\leq a\leq\frac{1}{2}$ and $\lambda>-a^{2}$. Then there exists $C>0$ such that
\begin{align*}
\int_{\mathbb{R}^{2}}\frac{e^{4\pi|u|^{2}}-1}{|x|^{2}}dx\leq C
\end{align*}
for any complex-valued function $u\in C^{\infty}_0(\mathbb{R}^{2}\setminus\{0\})$ with
\begin{align*}
\int_{\mathbb{R}^{2}}|\nabla_{\mathbf{A}}u|^{2}dx+\lambda
\int_{\mathbb{R}^{2}}\frac{|u|^{2}}{|x|^{2}}dx\leq 1.
\end{align*}
Moreover, the constant $4\pi$ is sharp.
\end{theorem}
\begin{remark} We remark that Theorem \ref{th1.2} implies a Trudinger-Moser inequality on $\mathbb{S}^{2}=\{x\in\mathbb{R}^{3}:|x|=1\}$ for Aharonov-Bohm magnetic    field.
As in \cite{bdel2}, section 2.3, we use cylindrical
coordinates $(\theta,t)\in[0,2\pi]\times[-1,1]$. Then $d\sigma=d\theta dt$ and
the magnetic gradient takes the form (see also \cite{bdel2}, section 2.3)
\begin{equation*}
  |\nabla_{\mathbf{A}}u|^{2}=(1-t^{2})|\partial_{t}u|^{2}+\frac{1}{1-t^{2}}|\partial_{\theta}u-iau|^{2}.
\end{equation*}
Substituting $r=(\frac{1+t}{1-t})^{\frac{1}{2}}$, we have
\begin{equation*}
  \begin{split}
\int_{\mathbb{S}^{2}} |\nabla_{\mathbf{A}}u|^{2}d\sigma=&\int_{0}^{2\pi}\int_{0}^{\infty}\left(
|\partial_{r}u|^{2}+\frac{1}{r^{2}}|\partial_{\theta}u-iau|^{2}\right)rd\theta dr;\\
\int_{\mathbb{S}^{2}} \frac{|u|^{2}}{1-t^{2}}d\sigma=&\int_{0}^{2\pi}\int_{0}^{\infty}\frac{|u|^{2}}{r}d\theta dr.
  \end{split}
\end{equation*}
Therefore, by Theorem \ref{th1.2}, we obtain, for $\lambda>0$,
\begin{equation*}
  \sup_{\int_{\mathbb{S}^{2}} |\nabla_{\mathbf{A}}u|^{2}d\sigma+\lambda\int_{\mathbb{S}^{2}} \frac{|u|^{2}}{1-t^{2}}d\sigma\leq 1}
  \int_{\mathbb{S}^{2}} \frac{e^{4\pi|u|^{2}}-1}{1-t^{2}}d\sigma<\infty.
\end{equation*}

\end{remark}

 As an application of Theorem \ref{th1.2}, we obtain the following  asymptotic estimates on best constants of magnetic Hardy-Sobolev inequalities (\ref{1.3}).
\begin{corollary}\label{co1.4}
 The optimal constant $\mu_{p}(\lambda)$ in (\ref{1.3}) satisfies
 \begin{align*}
\lim_{p\rightarrow\infty}p\mu_{p}(\lambda)=8\pi e.
 \end{align*}

\end{corollary}

We will also establish a Hardy-Trudinger-Moser inequality  in $\mathbb{B}^{2}=\{x\in\mathbb{R}^{2}: |x|<1\}$
which improves  the classical Trudinger-Moser inequality on the ball via  a subtraction of an additional
Hardy term. In fact, we have the following theorem.
\begin{theorem} \label{th1.3}
Let $0\leq a\leq\frac{1}{2}$ and $\lambda>-a^{2}$. Then there exists some $C>0$ such that
\begin{align*}
\int_{\mathbb{B}^{2}}\frac{e^{4\pi|u|^{2}}-1}{|x|^{2}}dx\leq C
\end{align*}
for any complex-valued function $u\in C^{\infty}_0(\mathbb{B}^{2}\setminus\{0\})$ with
\begin{align*}
\int_{\mathbb{B}^{2}}|\nabla_{\mathbf{A}}u|^{2}dx+\lambda
\int_{\mathbb{B}^{2}}\frac{|u|^{2}}{|x|^{2}}dx-\frac{1}{4}\int_{\mathbb{B}^{2}}\frac{|u|^{2}}{|x|^{2}\ln^{2}|x|}dx\leq 1.
\end{align*}
Moreover, the constant $4\pi$ is sharp.
\end{theorem}

We mention that classical Hardy-Trudinger-Moser type inequalities and their higher order version of Hardy-Adams inequalities have been extensively studied, see e.g., \cite{ad, ks, lly, lly2, ly2, ly, ly3, ly4, mst, MWY-ANS, LLWY, wy}

The paper is organized as follows. In Section 2, we will review some preliminary  facts which will be needed.
Section 3 focuses on the proof of Theorem \ref{th1.2} in the case of $a=0$. The proof of the case $0<a\leq\frac{1}{2}$ will be given in Section 4.
  Section 4 also offers the proof of  Theorem \ref{th1.3}.

\section{Notations and preliminaries}
We begin by quoting some preliminary facts which will be needed in
the sequel. In what follows, $a \lesssim b$ or $a=O(b)$ will stand for $a\leq Cb$ with a positive constant $C$ and $a \thicksim b$ stand for $C^{-1}b\leq a\leq Cb$.

\subsection{Hypergeometric function}
The hypergeometric function $F(a,b;c;z)$ is defined  by the power series
  \begin{equation*}%\label{2.4}
  \begin{split}
F(a,b;c;z)=\sum^{\infty}_{k=0}\frac{(a)_{k}(b)_{k}}{(c)_{k}}\frac{z^{k}}{k!},
\end{split}
\end{equation*}
where $c\neq0,-1,\cdots,-n,\cdots$ and $(a)_{k}$ is the rising Pochhammer symbol defined by
$$
(a)_{0}=1,\;(a)_{k}=a(a+1)\cdots(a+k-1), \;k\geq1.
$$

 Here, we
only list some of properties of hypergeometric function which will be used in the rest of paper. For
more information about  these functions, we refer to \cite{gr}, section 9.1 and \cite{er}, Chapter II.
\begin{itemize}
  \item Integral representation:
  \begin{equation}\label{2.1}
  \begin{split}
F(a,b;c;z)=\frac{\Gamma(c)}{\Gamma(c-b)\Gamma(b)}\int_{0}^{1}t^{b-1}(1-t)^{c-b-1}(1-tz)^{-a}dt,\;  c>b>0.
\end{split}
\end{equation}

  \item Transformation formula:
    \begin{equation}\label{2.2}
  \begin{split}
F(a,b;c;z)=(1-z)^{c-a-b} F(c-a,c-b;c;z).
\end{split}
\end{equation}

\end{itemize}

\subsection{Modified Bessel Function of the Second Kind}
Denote by $K_{\nu}(z)$ the modified Bessel function of the second kind. The following integral representation formula and asymptotic formulas for $K_{\nu}$ can be found in the \cite{gr} Page 368 and \cite{as}, Section 9.6:
\begin{align}\label{2.3}
&\int_{0}^{\infty}x^{\nu-1}e^{-\frac{\beta}{x}-\gamma x}dx=2\left(\frac{\beta}{\gamma}\right)^{\nu/2}K_{\nu}(2\sqrt{\beta\gamma}),\;\;\textrm{Re}\beta>0,\;
\textrm{Re}\gamma>0;\\
\label{2.4}
&K_{0}(z)\thicksim -\ln z,\;z\rightarrow 0+;\;\;\;K_{\nu}(z)\thicksim z^{-\nu},\;\nu>0,\; z\rightarrow 0+;\\
\label{2.5}
&K_{\nu}(z)\thicksim z^{-\frac{1}{2}}e^{-z},\;\; \nu\geq0,\; z\rightarrow\infty.
\end{align}
In particular (see \cite{gr}, Page 925, 8.468),
\begin{align}\label{2.6}
K_{1/2}(z)=\sqrt{\frac{\pi}{2}}z^{-1/2}e^{-z}.
\end{align}

 On the other hand, by (\ref{2.3}), we  have
\begin{align*}
 2\left(\frac{\beta}{\gamma}\right)^{\nu/2}K_{\nu}(2\sqrt{\beta\gamma})\leq \int_{0}^{\infty}x^{\nu-1}e^{-\gamma x}dx=\frac{\Gamma(\nu)}{\gamma^{\nu}}, \beta>0,\gamma>0,\nu>0,
\end{align*}
i.e.
\begin{align}\label{2.7}
K_{\nu}(z)\leq2^{\nu-1}\Gamma(\nu)z^{-\nu},\;\; z>0,\;\nu>0.
\end{align}

\subsection{O'Neil's lemma} Let $f$ be a measurable function on a measure space $(X,\mu)$. Denote by $m(f,s)=\mu (E_{s})$, where $E_{s}=\{x: |f(x)|>s\}$.
The
non-increasing rearrangement of $f$ is the function $f^{\ast}$ on $[0,\infty)$ defined by
\begin{equation*}
  f^{\ast}(t)=\inf\{s\in[0,\infty): m(f,s)\leq t \}.
\end{equation*}
If we  let
\begin{equation*}
  f^{\ast\ast}(t)=\frac{1}{t}\int_{0}^{t}f^{\ast}(s)ds,\;\; t>0,
\end{equation*}
then
\begin{equation}\label{2.8}
  tf^{\ast\ast}(t)=xf^{\ast}(t)+\int_{f^{\ast}(t)}^{\infty}m(f,s)ds.
\end{equation}
Moreover, there holds (see \cite {on})
\begin{equation}\label{2.9}
  (f_{1}+f_{2})^{\ast\ast}(t)\leq f_{1}^{\ast\ast}(t)+f_{2}^{\ast\ast}(t),\;\; t>0.
\end{equation}

Given three measure spaces $(X_{1},\mu_{1})$, $(X_{2},\mu_{2})$ and $(X_{3},\mu_{3})$. Let $T: (X_{1},\mu_{1})\times (X_{2},\mu_{2})\rightarrow (X_{3},\mu_{3})$ be a convolution operator. Set $h=T(f,g)$. It has been shown by O'Neil (see \cite {on}, Page 133) that the following inequality holds
\begin{equation}\label{2.10}
  h^{\ast\ast}(t)\leq -\int_{t}^{\infty}s g^{\ast\ast}(s)d f^{\ast}(s)+g^{\ast\ast}(t)\int_{f^{\ast}(t)}^{\infty}m(f,s)ds,\;\; t>0.
\end{equation}
Moreover,  integrating by parts yields
\begin{equation}\label{2.11}
  \begin{split}
  h^{\ast\ast}(t)\leq& -s\bigl.g^{\ast\ast}(s) f^{\ast}(s)\bigr|_{t}^{\infty}+\int_{t}^{\infty}f^{\ast}(s) g^{\ast}(s)ds+
  g^{\ast\ast}(t)\int_{f^{\ast}(t)}^{\infty}m(f,s)ds\\
  \leq& tg^{\ast\ast}(t) f^{\ast}(t)+ g^{\ast\ast}(t)\int_{f^{\ast}(t)}^{\infty}m(f,s)ds+\int_{t}^{\infty}f^{\ast}(s) g^{\ast}(s)ds\\
=&tg^{\ast\ast}(t)\left(f^{\ast}(t)+\int_{f^{\ast}(t)}^{\infty}m(f,s)ds\right)+\int_{t}^{\infty}f^{\ast}(s) g^{\ast}(s)ds\\
  =&tg^{\ast\ast}(t) f^{\ast\ast}(t)+\int_{t}^{\infty}f^{\ast}(s) g^{\ast}(s)ds.
  \end{split}
\end{equation}

\subsection{Heat kernel on $\mathbb{S}^{1}$}
Let $\mathbb{S}^{1}=\{e^{i\theta}: -\pi \leq \theta<\pi\}$ be the unit circle in $\mathbb{R}^{2}$.
Denote by $\Delta_{\mathbb{S}^{1}}=\partial_{\theta}^{2}$ the Laplace-Beltrami operator on $\mathbb{S}^{1}$. It is known that
the heat kernel on  $\Bbb{S}^1$ has the following explicit formula (see  e.g.  \cite{f}, Section 8.5):
\begin{equation}\label{2.12}
  e^{t \Delta_{\mathbb{S}^1}} =  \frac{1}{2\pi}\sum_{n \in \Bbb{Z}} e^{-n^{2}t}e^{in(\theta-\theta')}.
\end{equation}
By using  the Poisson summation formula, we have (see also \cite{ta}, Page 116, (4.25))
\begin{equation}\label{2.13}
e^{t \Delta_{\mathbb{S}^1}} = \frac{1}{\sqrt{4\pi t}} \sum_{n \in \Bbb{Z}} e^{- \frac{(\theta-\theta' - 2n \pi)^2}{4t}}.
\end{equation}
Here we use the fact
\begin{align}\label{2.14}
\int^{\infty}_{-\infty}e^{i\theta\xi}e^{-\xi^{2}t}d\xi=\sqrt{\frac{\pi}{t}}e^{-\frac{\theta^{2}}{4t}},\; t>0.
\end{align}

\section{Proof of Theorem \ref{th1.2}: case $a=0$}

If we substitute $w=\ln r$, then
\begin{align*}
\int_{\mathbb{R}^{2}}|\nabla_{\mathbf{A}}u|^{2}dx=&\int^{\infty}_{0}\int_{0}^{2\pi}\left(\left|\frac{\partial u}{\partial r}\right|^{2}+\frac{1}{r^{2}}
\left|\left(\frac{\partial }{\partial \theta}-ia\right)u\right|^{2}\right)rdrd\theta\\
=&\int_{-\infty}^{\infty}\int_{0}^{2\pi}\left(\left|\frac{\partial u}{\partial w}\right|^{2}+
\left|\left(\frac{\partial }{\partial \theta}-ia\right)u\right|^{2}\right)dwd\theta;\\
\int_{\mathbb{R}^{2}}\frac{e^{4\pi|u|^{2}}-1}{|x|^{2}}dx=&\int_{-\infty}^{\infty}\int_{0}^{2\pi}(e^{4\pi|u|^{2}}-1)dwd\theta.
\end{align*}
Therefore, in case $a=0$, Theorem \ref{th1.2} is equivalent  to the following Trudinger-Moser inequality on $\mathbb{R}\times \mathbb{S}^{1}$:
\begin{theorem}\label{th3.1}
Let $\lambda>0$. There exists $C>0$ such that
\begin{align*}
\int_{\mathbb{R}\times \mathbb{S}^{1}}(e^{4\pi|u|^{2}}-1)dwd\theta\leq C
\end{align*}
for any complex-valued function $u\in C^{\infty}_0(\mathbb{R}\times \mathbb{S}^{1})$ with
\begin{align*}
\int_{\mathbb{R}\times \mathbb{S}^{1}}(|\partial_{w}u|^{2}+|\partial_{\theta}u|^{2}+\lambda|u|^{2})dwd\theta\leq 1.
\end{align*}
Moreover, the constant $4\pi$ is sharp.
\end{theorem}

In the rest of this section, we shall prove Theorem \ref{th3.1}.

Notice that the heat kernel for $-\partial_{w}^{2}-\Delta_{\mathbb{S}^{1}}-\lambda$ is
\begin{equation*}
  e^{t(\partial_{w}^{2}+\Delta_{\mathbb{S}^{1}}+\lambda)}=e^{-\lambda t}\frac{1}{4\pi t}e^{-\frac{|w-w'|^{2}}{4t}}\sum_{n \in \Bbb{Z}} e^{- \frac{(\theta-\theta' - 2n \pi)^2}{4t}}.
\end{equation*}
 We have the  kernel of the fractional powers (see e.g. \cite{ben})
\begin{align*}
(-\partial_{w}^{2}-\Delta_{\mathbb{S}^{1}}-\lambda)^{-1/2}=&\frac{1}{\Gamma(1/2)}\int_{0}^{\infty}t^{-1/2}e^{t(\partial_{w}^{2}+\Delta_{\mathbb{S}^{1}}+\lambda)}dt\\
=&\frac{1}{4\pi^{3/2}}\sum_{n \in \Bbb{Z}}\int_{0}^{\infty}t^{-3/2}e^{-\lambda t-\frac{|w-w'|^{2}+(\theta-\theta' - 2n \pi)^2}{4t}}dt.
\end{align*}
Substituting $t$ by $\frac{1}{t}$ and using (\ref{2.3}), we obtain
\begin{equation}\label{3.1}
  \begin{split}
  &(-\partial_{w}^{2}-\Delta_{\mathbb{S}^{1}}-\lambda)^{-1/2}\\
=&\frac{1}{4\pi^{3/2}}\sum_{n \in \Bbb{Z}}\int_{0}^{\infty}t^{-1/2}e^{-\frac{\lambda}{t} -\frac{|w-w'|^{2}+(\theta-\theta' - 2n \pi)^2}{4}t}dt
\\
=&\frac{1}{4\pi^{3/2}} \sum_{n \in \mathbb{Z}}2\left(\frac{4\lambda}{w^{2}+(\theta-\theta' - 2n \pi)^2}\right)^{\frac{1}{4}}K_{1/2}(\sqrt{\gamma((w-w')^{2}+(\theta-\theta' - 2n \pi)^2)})\\
=&\frac{1}{2\pi} \sum_{n \in \mathbb{Z}}\frac{1}{\sqrt{(w-w')^{2}+(\theta-\theta' - 2n \pi)^2}}e^{-\sqrt{\gamma((w-w')^{2}+(\theta-\theta' - 2n \pi)^2)}}.
  \end{split}
\end{equation}
To get the last equation, we use (\ref{2.6}).

For simplicity, we set $\phi_{1}(w-w',\theta-\theta')=(-\partial_{w}^{2}-\Delta_{\mathbb{S}^{1}}-\lambda)^{-1/2}$. By (\ref{3.1}), it is easy to see that there exit $A_{1}>0$ and $\delta_{1}>0$ such that
 \begin{align}\label{b3.1}
\phi_{1}(w-w',\theta-\theta')\leq & \frac{1}{2\pi} \frac{1}{\sqrt{(w-w')^{2}+(\theta-\theta')^2}}+A_{1},\;\;|w-w'|\leq 1,\;\;|\theta-\theta'|\leq \pi;\\
\label{b3.2}
\phi_{1}(w-w',\theta-\theta')\lesssim & e^{-\delta_{1}|w-w'|},\;\;\;\;\;\;\;\;\;\;\;\;\;\;\;\;\;\;\;\;\;\; \;\;\;\;\;\;\;\;
\;\;\;\;\;\;\;|w-w'|\geq 1,\;\;|\theta-\theta'|\leq \pi.
 \end{align}
Therefore, the non-increasing rearrangement of $\phi_{1}$ satisfies
\begin{equation}\label{3.2}
\begin{split}
\phi_{1}^{\ast}(t)\leq& \frac{1}{\sqrt{4\pi t}}+A_{2},\;\; 0<t\leq 1;\\
 \phi_{1}^{\ast}(t)\lesssim& e^{-\delta_{2}t},\;\;\;\;\;\;\;\;\;\;\;\;\; t>1,
\end{split}
\end{equation}
where $A_{2}>0$ and $\delta_{2}>0$.

Now we prove Theorem \ref{1.3}. The main idea is to adapt the level set developed by Lam and the first author to derive a global Trudinger-Moser inequality
from a local one (see \cite{ll,l}). \\
\

\textbf{Proof of Theorem \ref{th3.1}}  Let $u\in C^{\infty}_{0}(\mathbb{R}\times \mathbb{S}^{1})$ be such that
\begin{equation*}
  \int_{\mathbb{R}\times \mathbb{S}^{1}}(|\partial_{w}u|^{2}+|\partial_{\theta}u|^{2}+\lambda|u|^{2})dwd\theta=\int_{\mathbb{R}\times \mathbb{S}^{1}}|(-\partial_{w}^{2}-\Delta_{\mathbb{S}^{1}}-\lambda)^{1/2}u|^{2}
dwd\theta\leq 1.
\end{equation*}
If we  set $\Omega(u)=\{(w,\theta)\in\mathbb{R}\times \mathbb{S}^{1}:|u(w,\theta)|\geq1\}$, then
\begin{equation*}%\label{3.3}
\begin{split}
|\Omega(u)|=&\int_{\Omega(u)}dwd\theta\leq\int_{\Omega(u)}|u|^{2}dwd\theta\leq\int_{\mathbb{R}\times \mathbb{S}^{1}}|u|^{2}dwd\theta\leq\lambda^{-1}.
\end{split}
\end{equation*}
We write
\begin{equation*}%\label{3.4}
\begin{split}
\int_{\mathbb{R}\times \mathbb{S}^{1}}(e^{4\pi|u|^{2}}-1)dwd\theta
=&\int_{\Omega(u)}(e^{4\pi|u|^{2}}-1)dwd\theta+
\int_{\mathbb{R}\times \mathbb{S}^{1}\setminus\Omega(u)}(e^{4\pi|u|^{2}}-1)dwd\theta\\
\leq&\int_{\mathbb{R}\times \mathbb{S}^{1}\setminus\Omega(u)}(e^{4\pi|u|^{2}}-1)dwd\theta+\int_{\Omega(u)}
e^{4\pi|u|^{2}}dwd\theta.
\end{split}
\end{equation*}
Notice that on the domain $\mathbb{R}\times \mathbb{S}^{1}\setminus\Omega(u)$, we have $|u(w,\theta)|<1$. Thus, we get
\begin{equation*}%\label{3.5}
\begin{split}
\int_{\mathbb{R}\times \mathbb{S}^{1}\setminus\Omega(u)}(e^{4\pi|u|^{2}}-1)dwd\theta=&\int_{\mathbb{R}\times \mathbb{S}^{1}\setminus\Omega(u)}\sum^{\infty}_{n=2}\frac{(4\pi |u|^{2})^{n}}{n!}dwd\theta\\
\leq&\int_{\mathbb{R}\times \mathbb{S}^{1}\setminus\Omega(u)}\sum^{\infty}_{n=1}\frac{(4\pi )^{n}|u|^{2}}{n!}dwd\theta\\
\leq&\sum^{\infty}_{n=1}\frac{(4\pi )^{n}}{n!}\int_{\mathbb{R}\times \mathbb{S}^{1}}|u|^{2}dwd\theta\\
\leq&\sum^{\infty}_{n=1}\frac{(4\pi )^{n}}{n!}\lambda^{-1}.
\end{split}
\end{equation*}

Next we shall  show $\int_{\Omega(u)}e^{4\pi |u|^{2}}dwd\theta$ is bounded by some universal constant.

Set $v=(-\partial_{w}^{2}-\Delta_{\mathbb{S}^{1}}-\lambda)^{1/2}u$. Then
\begin{align*}
|u(w,\theta)|=|(-\partial_{w}^{2}-\Delta_{\mathbb{S}^{1}}-\lambda)^{-1/2}v|
\leq&\int_{\mathbb{R}\times \mathbb{S}^{1}}|v(w',\theta')|\phi_{1}(w-w',\theta-\theta')dw'd\theta',
\end{align*}
where $\phi_{1}$ satisfies (see (\ref{3.2}))
\begin{equation*}%\label{3.2}
 \phi_{1}^{\ast}(t)\leq \frac{1}{\sqrt{4\pi t}}+O(1),\;\; 0<t\leq 1,\;\; \int_{c}^{\infty}|\phi_{1}^{\ast}(t)|^{2}<\infty,\;\; \forall c>0.
\end{equation*}
Closely following the proof of Theorem 1.7 in \cite{lly}, we have that there exists a constant
C which is independent of $u$ and $\Omega(u)$ such that
\begin{equation*}
  \int_{\Omega(u)}e^{4\pi |u|^{2}}dwd\theta=\int_{0}^{|\Omega(u)|}e^{4\pi (|u|^{\ast}(t))^{2}}dt\leq
  \int_{0}^{1/\lambda}e^{4\pi (|u|^{\ast}(t))^{2}}dt<C.
\end{equation*}
Therefore,
\begin{align*}
\int_{\mathbb{R}\times \mathbb{S}^{1}}(e^{4\pi|u|^{2}}-1)dwd\theta
\leq&\int_{\Omega(u)}(e^{4\pi|u|^{2}}-1)dwd\theta+
\int_{\mathbb{R}\times \mathbb{S}^{1}\setminus\Omega(u)}e^{4\pi|u|^{2}}dwd\theta\\
\leq& \sum^{\infty}_{n=1}\frac{(4\pi )^{n}}{n!}\lambda^{-1}+C.
\end{align*}

Finally, we show that the constant $4\pi$ is sharp. Define, for $0<\delta<1$,
\begin{equation}\label{3.3}
  u_{\delta}(w,\theta)=\left\{
                         \begin{array}{ll}
                          - \ln \delta, & \hbox{$w^{2}+\theta^{2}<\delta^{2}$;} \\
                           -\frac{1}{2}\ln (w^{2}+\theta^{2}), & \hbox{$\delta^{2}\leq w^{2}+\theta^{2}\leq 1$;} \\
                           0, & \hbox{$w^{2}+\theta^{2}>1$.}
                         \end{array}
                       \right.
\end{equation}
 A routine calculation shows
\begin{equation}\label{3.4}
  \begin{split}
   \int_{\mathbb{R}\times \mathbb{S}^{1}}(|\partial_{w}u_{\delta}|^{2}+|\partial_{\theta}u_{\delta}|^{2}+\lambda|u_{\delta}|^{2})dwd\theta
=&-2\pi\ln\delta+\lambda\pi\left(\frac{1}{2}-\frac{1}{2}\delta^{2}+\delta^{2}\ln\delta\right).
  \end{split}
\end{equation}
Set
\begin{equation*}%\label{b3.3}
  \widetilde{u}_{\delta}(w,\theta)=\frac{1}{\sqrt{-2\pi\ln\delta+\lambda\pi\left(\frac{1}{2}-\frac{1}{2}\delta^{2}+\delta^{2}\ln\delta\right)}}u_{\delta}(w,\theta).
\end{equation*}
Then
$$\int_{\mathbb{R}\times \mathbb{S}^{1}}(|\partial_{w}\widetilde{u}_{\delta}|^{2}+|\partial_{\theta}\widetilde{u}_{\delta}|^{2}+\lambda|\widetilde{u}_{\delta}|^{2})dwd\theta=1.$$

Suppose now that
\begin{align*}
\int_{\mathbb{R}\times \mathbb{S}^{1}}(e^{\beta|\widetilde{u}_{\delta}|^{2}}-1)dwd\theta\leq C.
\end{align*}
Then
\begin{align*}
&\int_{\{(w,\theta): w^{2}+\theta^{2}<\delta^{2}\}}(e^{4\pi|\widetilde{u}_{\delta}|^{2}}-1)dwd\theta\\
=&\left[\exp\left(\frac{\beta \ln^{2}\delta}{-2\pi\ln\delta+\lambda\pi\left(\frac{1}{2}-\frac{1}{2}\delta^{2}+\delta^{2}\ln\delta\right)}\right)-1\right]\pi \delta^{2}\leq C.
\end{align*}
Therefore,
\begin{align*}
\beta\leq \ln\left(\frac{C}{\pi \delta^{2}}+1\right)\times
\frac{-2\pi\ln\delta+\lambda\pi\left(\frac{1}{2}-\frac{1}{2}\delta^{2}+\delta^{2}\ln\delta\right)}{ \ln^{2}\delta}\rightarrow 4\pi,\;\;\delta\rightarrow0+.
\end{align*}
This completes the proof of Theorem \ref{th3.1}.

\section{Proof of Theorem \ref{th1.2}: case $0<a\leq \frac{1}{2}$}

For $0<a\leq \frac{1}{2}$ and $\epsilon>0$, we define
\begin{equation*}
  T_{a}=-\sum_{n \in \Bbb{Z}}\left(n^{2}-2a\frac{n^{2}}{\sqrt{n^{2}+\epsilon}}\right)P_{n},
\end{equation*}
where $P_{n}: L^{2}(\mathbb{S}^{1})\rightarrow \mathbb{C}e^{in\theta} (n\in\mathbb{Z})$ is the orthogonal projection.
It is easy to verify
\begin{equation*}
  n^{2}-2a\frac{n^{2}}{\sqrt{n^{2}+\epsilon}}\geq 0,\;\;\; 0\leq a\leq\frac{1}{2},\;\; n\in\mathbb{Z},
\end{equation*}
and
\begin{equation*}
0\leq  -\int_{\mathbb{S}^{1}}T_{a}f(\theta) \overline{f}(\theta)d\theta\leq \int_{\mathbb{S}^{1}}|\partial_{\theta}f|^{2}d\theta,\;\;f\in C^{1}(\mathbb{S}^{1}).
\end{equation*}

 Since $P_{n}$ is  a self-adjoint operator, so does
$T_{a}$.
Moreover, the  heat kernels for $T_{a}$ is given by
\begin{equation*}
e^{tT_{a}}=\frac{1}{2\pi}\sum_{n \in \Bbb{Z}}e^{-\left(n^{2}-2a\frac{n^{2}}{\sqrt{n^{2}+\epsilon}}\right)t}e^{in(\theta-\theta')}
\end{equation*}
In order to use  Poisson summation formula (see \cite{sw}, Page 252, Corollary 2.6), we need the following lemma:
\begin{lemma}\label{lm4.1}
Let $a>0$, $\epsilon>0$ and  $t>0$. There exists some constant $C>0$ which is independent of $t$ and $\theta$ such that
\begin{align}\label{4.1}
\left|\int^{\infty}_{-\infty}e^{i\theta\xi}e^{-\left(\xi^{2}-2a\frac{\xi^{2}}{\sqrt{\xi^{2}+\epsilon}}\right)t}d\xi\right|\leq& \frac{1}{\sqrt{t}} e^{a^{2}t},\;\;
\;\;\;\;\;\;\;\;\;\;\;\;\;\;\;\forall\theta\in \mathbb{R};\\
\label{4.2}
\left|\int^{\infty}_{-\infty}e^{i\theta\xi}e^{-\left(\xi^{2}-2a\frac{\xi^{2}}{\sqrt{\xi^{2}+\epsilon}}\right)t}d\xi\right|\leq& \frac{1}{\theta^{2}} \sqrt{t}(1+t)e^{a^{2}t},\;\;\;\;\forall|\theta|>0.
\end{align}
\end{lemma}
\begin{proof}
We have, for $t>0$,
\begin{equation}\label{4.3}
  \begin{split}
  \left|\int^{\infty}_{-\infty}e^{i\theta\xi}e^{-\left(\xi^{2}-2a\frac{\xi^{2}}{\sqrt{\xi^{2}+\epsilon}}\right)t}d\xi\right|\leq&
\int^{\infty}_{-\infty}e^{-\left(\xi^{2}-2a\frac{\xi^{2}}{\sqrt{\xi^{2}+\epsilon}}\right)t}d\xi\\
=&2\int^{\infty}_{0}e^{-\left(\xi^{2}-2a\frac{\xi^{2}}{\sqrt{\xi^{2}+\epsilon}}\right)t}d\xi\\
\leq&2\int^{\infty}_{0}e^{-\left(\xi^{2}-2a\xi\right)t}d\xi\\
\leq&2\int^{\infty}_{-\infty}e^{-\left(\xi^{2}-2a\xi\right)t}d\xi\\
=&\frac{2}{\sqrt{t}}e^{a^{2}t}\int^{\infty}_{-\infty}e^{-\xi^{2}}d\xi.
  \end{split}
\end{equation}
This proves (\ref{4.1}).

Next we shall prove (\ref{4.2}). By integrating by parts, we get
\begin{equation}\label{4.4}
  \begin{split}
&\int^{\infty}_{-\infty}e^{i\theta\xi}e^{-\left(\xi^{2}-2a\frac{\xi^{2}}{\sqrt{\xi^{2}+\epsilon}}\right)t}d\xi\\
=&\frac{1}{i\theta}
\int^{\infty}_{-\infty}e^{-\left(\xi^{2}-2a\frac{\xi^{2}}{\sqrt{\xi^{2}+\epsilon}}\right)t}de^{i\theta\xi}\\
=&-\frac{1}{i\theta}\int^{\infty}_{-\infty}e^{i\theta\xi}e^{-\left(\xi^{2}-2a\frac{\xi^{2}}{\sqrt{\xi^{2}+\epsilon}}\right)t}
\left[-2t\xi+4at\frac{\xi}{(\xi^{2}+\epsilon)^{\frac{1}{2}}} -2at\frac{\xi^{3}}{(\xi^{2}+\epsilon)^{\frac{3}{2}}} \right]d\xi\\
=&\frac{1}{\theta^{2}}\int^{\infty}_{-\infty}e^{-\left(\xi^{2}-2a\frac{\xi^{2}}{\sqrt{\xi^{2}+\epsilon}}\right)t}
\left[-2t\xi+4at\frac{\xi}{(\xi^{2}+\epsilon)^{\frac{1}{2}}} -2at\frac{\xi^{3}}{(\xi^{2}+\epsilon)^{\frac{3}{2}}} \right]de^{i\theta\xi}\\
=&-\frac{1}{\theta^{2}}\int^{\infty}_{-\infty}e^{i\theta\xi}e^{-\left(\xi^{2}-2a\frac{\xi^{2}}{\sqrt{\xi^{2}+\epsilon}}\right)t}
\left[-2t\xi+4at\frac{\xi}{(\xi^{2}+\epsilon)^{\frac{1}{2}}} -2at\frac{\xi^{3}}{(\xi^{2}+\epsilon)^{\frac{3}{2}}} \right]^{2}d\xi+\\
&\frac{1}{\theta^{2}}\int^{\infty}_{-\infty}e^{i\theta\xi}e^{-\left(\xi^{2}-2a\frac{\xi^{2}}{\sqrt{\xi^{2}+\epsilon}}\right)t}
\left[-2t+\frac{4a\epsilon t}{(\xi^{2}+\epsilon)^{\frac{3}{2}}}
-\frac{6at\xi^{2}}{(\xi^{2}+\epsilon)^{\frac{5}{2}}} \right]d\xi.
  \end{split}
\end{equation}
Therefore,
\begin{align}\label{4.5}
\left|\int^{\infty}_{-\infty}e^{i\theta\xi}e^{-\left(\xi^{2}-2a\frac{\xi^{2}}{\sqrt{\xi^{2}+\epsilon}}\right)t}d\xi\right|\lesssim&\frac{1}{\theta^{2}}
(t+t^{2})\int^{\infty}_{-\infty}e^{-\left(\xi^{2}-2a\frac{\xi^{2}}{\sqrt{\xi^{2}+\epsilon}}\right)t}d\xi.
\end{align}
With the same argument as in (\ref{4.3}), we obtain
\begin{align*}
\left|\int^{\infty}_{-\infty}e^{i\theta\xi}e^{-\left(\xi^{2}-2a\frac{\xi^{2}}{\sqrt{\xi^{2}+\epsilon}}\right)t}d\xi\right|\lesssim&\frac{1}{\theta^{2}}
(1+t)\sqrt{t}e^{a^{2}t}.
\end{align*}
This completes the proof of  Lemma \ref{lm4.1}.
\end{proof}

Therefore, by Poisson summation formula, we have
\begin{equation*}
e^{tT_{a}}=\frac{1}{2\pi}\sum_{n \in \Bbb{Z}}\int^{\infty}_{-\infty}e^{i(\theta-\theta'-2n\pi)\xi}e^{-\left(\xi^{2}-2a\frac{\xi^{2}}{\sqrt{\xi^{2}+\epsilon}}\right)t}d\xi, \; \; t>0.
\end{equation*}

\begin{lemma}\label{lm4.2}
Set $\epsilon_{1}=1-\frac{2a}{\sqrt{1+\epsilon}}>0$. There exists   a constant $C=C(\epsilon,a)>0$ such that
\begin{align*}
\left|e^{tT_{a}}-e^{t\Delta_{\mathbb{S}^{1}}}\right|\leq C(1+t) e^{-\epsilon_{1} t},\;\;  t>0,\; |\theta-\theta'|\leq \pi.
\end{align*}
\end{lemma}
\begin{proof}
We have
\begin{align*}
\left|e^{tT_{a}}-e^{t\Delta_{\mathbb{S}^{1}}}\right|=&
\left|
\frac{1}{2\pi}\sum_{n \in \Bbb{Z}}\int^{\infty}_{-\infty}e^{i(\theta-\theta'-2n\pi)\xi}e^{-\left(\xi^{2}-2a\frac{\xi^{2}}{\sqrt{\xi^{2}+\epsilon}}\right)t}d\xi-
\frac{1}{\sqrt{4\pi t}} \sum_{n \in \Bbb{Z}} e^{- \frac{(\theta-\theta' - 2n \pi)^2}{4t}}\right|\\
=&\frac{1}{2\pi}\left|
\sum_{n \in \Bbb{Z}}\int^{\infty}_{-\infty}e^{i(\theta-\theta'-2n\pi)\xi}e^{-\left(\xi^{2}-2a\frac{\xi^{2}}{\sqrt{\xi^{2}+\epsilon}}\right)t}d\xi-
\sum_{n \in \Bbb{Z}}\int^{\infty}_{-\infty}e^{i(\theta-\theta'-2n\pi)\xi}e^{-\xi^{2}t}d\xi\right|\\
\leq&\frac{1}{2\pi}\sum_{n \in \Bbb{Z}}\left|\int^{\infty}_{-\infty}e^{i(\theta-\theta'-2n\pi)\xi}e^{-\xi^{2}t}
\left(e^{2a\frac{\xi^{2}}{\sqrt{\xi^{2}+\epsilon}}t}-1 \right)d\xi\right|.
\end{align*}

For $n=0$, we have
\begin{align*}
&\left|\int^{\infty}_{-\infty}e^{i(\theta-\theta')\xi}e^{-\xi^{2}t}
\left(e^{2a\frac{\xi^{2}}{\sqrt{\xi^{2}+\epsilon}}t}-1 \right)d\xi\right|\\
\leq& \int^{\infty}_{-\infty}e^{-\xi^{2}t}
\left(e^{2a\frac{\xi^{2}}{\sqrt{\xi^{2}+\epsilon}}t}-1 \right)d\xi\\
\leq&\int^{\infty}_{-\infty} e^{-\xi^{2}t}e^{2a\frac{\xi^{2}}{\sqrt{\xi^{2}+\epsilon}}t}a\frac{\xi^{2}}{\sqrt{\xi^{2}+\epsilon}}td\xi \;\;\;(\because\;e^{\alpha}-1\leq \alpha e^{\alpha},\;\forall\alpha\geq0)\\
\leq&at\int^{\infty}_{0} e^{-(\xi^{2}-2a\xi)t}\xi d\xi.
\end{align*}
On the other hand,
\begin{align*}
t\int^{\infty}_{0} e^{-(\xi^{2}-2a\xi)t}\xi d\xi=&e^{a^{2}t}t\int^{\infty}_{-a} e^{-\xi^{2}t}(\xi+a) d\xi\\
=&e^{a^{2}t}t\int^{\infty}_{-a} e^{-\xi^{2}t}\xi d\xi+ae^{a^{2}t}t\int^{\infty}_{-a} e^{-\xi^{2}t}d\xi\\
=&\frac{1}{2}+ae^{a^{2}t}\sqrt{t}\int^{\infty}_{a\sqrt{t}} e^{-\xi^{2}}d\xi\\
\leq& \frac{1}{2}+ae^{a^{2}t}\sqrt{t}\int^{\infty}_{0} e^{-\xi^{2}}d\xi.
\end{align*}
Therefore,
\begin{align}\label{4.6}
\left|\int^{\infty}_{-\infty}e^{i\theta\xi}e^{-\xi^{2}t}
\left(e^{2a\frac{\xi^{2}}{\sqrt{\xi^{2}+\epsilon}}t}-1 \right)d\xi\right|\leq&
\frac{1}{2}a+a^{2}e^{a^{2}t}\sqrt{t}\int^{\infty}_{0} e^{-\xi^{2}}d\xi\lesssim 1,\;\; 0<t\leq1.
\end{align}

If $n\neq 0$,  then with the same argument as in (\ref{4.4})-(\ref{4.5}), we get
\begin{equation}\label{4.7}
\begin{split}
&\sum_{n \neq0}\left|\int^{\infty}_{-\infty}e^{i(\theta-\theta'-2\pi)\xi}e^{-\xi^{2}t}
\left(e^{2a\frac{\xi^{2}}{\sqrt{\xi^{2}+\epsilon}}t}-1 \right)d\xi\right|\\
\leq&
(t+t^{2})e^{a^{2}t}\int^{\infty}_{-\infty}e^{-\xi^{2}t}
\left(e^{2a\frac{\xi^{2}}{\sqrt{\xi^{2}+\epsilon}}t}-1 \right)d\xi\sum_{n \neq0}\frac{1}{(\theta-\theta'-2n\pi)^{2}}\\
\lesssim&(1+t)\sqrt{t}e^{a^{2}t}\sum_{n \neq 0}\frac{1}{(\theta-\theta'-2n\pi)^{2}}\\
\lesssim& 1,\;\;\;0<t\leq1,\;\;|\theta-\theta'|\leq \pi.
\end{split}
\end{equation}
Combining (\ref{4.6}) and (\ref{4.7}) yields
\begin{equation*}
  \left|e^{tT_{a}}-e^{t\Delta_{\mathbb{S}^{1}}}\right|\lesssim 1,\;\; 0<t\leq1, \;\;|\theta-\theta'|\leq \pi.
\end{equation*}

For $t>1$, we have
\begin{equation}\label{b4.7}
\begin{split}
\left|e^{tT_{a}}-e^{t\Delta_{\mathbb{S}^{1}}}\right|=&\left|\frac{1}{2\pi}\sum_{n \in \Bbb{Z}}e^{-\left(n^{2}-2a\frac{n^{2}}{\sqrt{n^{2}+\epsilon}}\right)t}e^{in\theta}
-\frac{1}{2\pi}\sum_{n \in \Bbb{Z}}e^{-n^{2}t}e^{in\theta}\right|\\
\leq&\frac{1}{2\pi}\sum_{n \neq 0}e^{-n^{2}t}\left(e^{2a\frac{n^{2}}{\sqrt{n^{2}+\epsilon}}t}-1\right)\\
\leq&\frac{1}{2\pi}\sum_{n \neq 0}e^{-n^{2}t}2a\frac{n^{2}}{\sqrt{n^{2}+\epsilon}}te^{2a\frac{n^{2}}{\sqrt{n^{2}+\epsilon}}t}\\
=& \frac{a}{\pi}t  \sum_{n \neq 0}\frac{n^{2}}{\sqrt{n^{2}+\epsilon}}e^{-\left(n^{2}-2a\frac{n^{2}}{\sqrt{n^{2}+\epsilon}}\right)t}\\
\lesssim& te^{(\frac{2a}{\sqrt{1+\epsilon}}-1)t}.
\end{split}
\end{equation}
The desired result follows by combining (\ref{4.7}) and (\ref{b4.7}).
\end{proof}

As an application of Lemma \ref{lm4.2}, we have the following corollary.
\begin{corollary}\label{co4.3}
There exists   a constant $C=C(\epsilon,a)>0$ such that
\begin{align*}
\left|e^{tT_{a}}-e^{t\Delta_{\mathbb{S}^{1}}}\right|\leq C,\;\;  t>0,\;\;|\theta-\theta'|\leq \pi.
\end{align*}
\end{corollary}

Now we can give asymptotic estimates of the kernel $(-\partial_{w}^{2}-T_{a}+\lambda)^{-\frac{1}{2}}$. The main result is the following:

\begin{lemma}\label{lm4.4}
Let  $\delta_{3}>0$  and set $\phi_{2}=(-\partial_{w}^{2}-T_{a}+\lambda)^{-\frac{1}{2}}$.
There holds, for some $\delta_{4}>0$,
\begin{align}\label{4.8}
|\phi_{2}|\leq& \frac{1}{2\pi} \frac{1}{\sqrt{(w-w')^{2}+(\theta-\theta')^2}}+O\left(\frac{1}{|w-w'|^{\delta_{3}}}\right),\;\;|w-w'|\leq 1,\; |\theta-\theta'|\leq \pi;\\
\label{4.9}
|\phi_{2}|\leq& e^{-\delta_{4}|w-w'|},\;\;\;\;\;\;\;\;\;\;\;\;\;\;\;\;\;\;\;\;\;\;\;\;\;\;\;\;\;\;\;\;\;\;\;\;\;\;\;\;\;\;\;\;\;\;\;\;\;\;\;\;\;\;
\;\;\;\;|w-w'|\geq 1,\; |\theta-\theta'|\leq \pi.
\end{align}
\end{lemma}
\begin{proof}
We have, by Corollary \ref{co4.3},
\begin{align*}
|\phi_{2}|=&\frac{1}{\Gamma(\frac{1}{2})}\left|\int_{0}^{\infty}t^{-\frac{1}{2}}e^{t(\partial_{w}^{2}+T_{a}-\lambda)}dt\right|\\
\leq& \frac{1}{\Gamma(\frac{1}{2})}\int_{0}^{\infty}t^{-\frac{1}{2}}e^{t(\partial_{w}^{2}+\Delta_{\mathbb{S}^{1}}-\lambda)}dt
+C\int_{0}^{\infty}t^{-\frac{1}{2}}e^{t(\partial_{w}^{2}-\lambda)}dt\\
=&(-\partial_{w}^{2}-\Delta_{\mathbb{S}^{1}}+\lambda)^{-\frac{1}{2}}+
C\int_{0}^{\infty}t^{-1}e^{-\lambda t-\frac{(w-w')^{2}}{4t}}dt\\
=&\phi_{1}(w-w',\theta-\theta')+C'K_{0}(2\sqrt{\lambda}|w-w'|),
\end{align*}
where $C$ and $C'$ are positive constant.
To get the last equality above, we use (\ref{2.3}).
By using (\ref{b3.1}) and (\ref{2.4}), we get,  for $|w-w'|\leq 1$ and $ |\theta-\theta'|\leq \pi$,
\begin{align*}
|\phi_{2}|\leq&\frac{1}{2\pi} \frac{1}{\sqrt{(w-w')^{2}+(\theta-\theta')^2}}+O\left(\ln\frac{e}{|w-w'|}\right)\\
\leq& \frac{1}{2\pi} \frac{1}{\sqrt{(w-w')^{2}+(\theta-\theta')^2}}+O\left(\frac{1}{|w-w'|^{\delta_{3}}}\right).
\end{align*}

On the other hand, if $|w-w'|\geq 1$, then by using (\ref{b3.2}) and (\ref{2.5}), we obtain, for some $C''>0$,
\begin{align*}
|\phi_{2}|\leq& e^{-\delta_{1}|w-w'|}+C''\frac{1}{\sqrt{|w-w'|}}e^{-2\sqrt{\lambda}|w-w'|}.
\end{align*}
This proves (\ref{4.9}).
The proof of Lemma \ref{lm4.4}
is thereby completed.
\end{proof}

Next we shall give asymptotic estimates of $\phi_{2}^{\ast}(t)$. We first need the following lemma:
\begin{lemma}\label{lm4.5} Let $\delta>0$ and
set $\phi_{3}=\frac{1}{|w-w'|^{\delta}}$. Then
\begin{equation*}
  \phi_{3}^{\ast}(t)= (2\pi)^{\delta}t^{-\delta},\;\;t>0.
\end{equation*}
\end{lemma}
\begin{proof}
Notice that
\begin{align*}
\phi_{3}^{\ast}(t)=\inf\{s: m(\phi_{3},s)\leq t\},
\end{align*}
where
$$m(\phi_{3},s)=\int_{\{(w,\theta): |w|^{-\delta}>s, |\theta|\leq \pi\}}dwd\theta=2\pi\int_{0}^{s^{-1/\delta}}ds=2\pi s^{-1/\delta}.$$
We have
\begin{align*}
\phi_{3}^{\ast}(t)=\inf\{s: 2\pi s^{-1/\delta}\leq t\}=(2\pi)^{\delta}t^{-\delta}.
\end{align*}
This completes the proof of Lemma \ref{lm4.5}.
\end{proof}

We remark that it's difficult to give the asymptotic estimates of $\phi_{2}^{\ast}(t)$ as $t\rightarrow0$. However,  using (\ref{2.9}), (\ref{4.8}) and Lemma \ref{lm4.5},
we obtain
\begin{equation}\label{4.10}
  \phi_{2}^{\ast\ast}(t)\leq \frac{1}{t}\int_{0}^{t}\frac{1}{\sqrt{4\pi s}}ds+O\left(
\frac{1}{t}\int_{0}^{t}s^{-\delta_{3}}ds\right)=\frac{1}{\sqrt{\pi t}}+O(t^{-\delta_{3}}),\;\;0<t\leq1.
\end{equation}
Here we use the fact that the non-increasing rearrangement of $\frac{1}{2\pi} \frac{1}{\sqrt{(w-w')^{2}+(\theta-\theta')^2}}$ satisfies
\begin{equation*}
  \left(\frac{1}{2\pi} \frac{1}{\sqrt{(w-w')^{2}+(\theta-\theta')^2}}\right)^{\ast}(t)\leq \frac{1}{\sqrt{4\pi t}},\;\; 0<t\leq 1.
\end{equation*}

For $t>1$, we have, by  using (\ref{4.9}),
\begin{equation}\label{4.11}
  \phi_{2}^{\ast\ast}(t)\lesssim e^{-\delta_{4} t},\;\;t>1,
\end{equation}
for some  $\delta_{4}>0$.

Following the proof of Theorem \ref{th3.1}, we have the following theorem.
\begin{theorem}\label{th4.6}
Let $\lambda>0$ and $0<a\leq\frac{1}{2}$. There exists $C>0$ such that
\begin{align*}
\int_{\mathbb{R}\times \mathbb{S}^{1}}(e^{4\pi|u|^{2}}-1)dwd\theta\leq C
\end{align*}
for any complex-valued function $u\in C^{\infty}_0(\mathbb{R}\times \mathbb{S}^{1})$ with
\begin{align*}
-\int_{\mathbb{R}\times \mathbb{S}^{1}}\overline{u} T_{a}u dwd\theta+\int_{\mathbb{R}\times \mathbb{S}^{1}}(|\partial_{w}u|^{2}+\lambda|u|^{2})dwd\theta\leq 1.
\end{align*}
\end{theorem}
\begin{proof}
As in the proof of  Theorem \ref{th3.1}, we  set
 $\Omega(u)=\{(w,\theta)\in\mathbb{R}\times \mathbb{S}^{1}:|u(w,\theta)|\geq1\}$ and write
\begin{equation*}%\label{3.4}
\begin{split}
&\int_{\mathbb{R}\times \mathbb{S}^{1}}(e^{4\pi|u|^{2}}-1)dwd\theta\\
=&\int_{\Omega(u)}(e^{4\pi|u|^{2}}-1)dwd\theta+
\int_{\mathbb{R}\times \mathbb{S}^{1}\setminus\Omega(u)}(e^{4\pi|u|^{2}}-1)dwd\theta\\
\leq&\int_{\mathbb{R}\times \mathbb{S}^{1}\setminus\Omega(u)}(e^{4\pi|u|^{2}}-1)dwd\theta+\int_{\Omega(u)}
e^{4\pi|u|^{2}}dwd\theta.
\end{split}
\end{equation*}
Then  $|\Omega(u)|\leq \lambda^{-1}$ and
\begin{equation*}%\label{3.4}
\begin{split}
\int_{\mathbb{R}\times \mathbb{S}^{1}}(e^{4\pi|u|^{2}}-1)dwd\theta
\leq&\sum^{\infty}_{n=1}\frac{(4\pi )^{n}}{n!}\lambda^{-1}+\int_{\Omega(u)}
e^{4\pi|u|^{2}}dwd\theta.
\end{split}
\end{equation*}
To finish the proof, it is enough to show $\int_{\Omega(u)}
e^{4\pi|u|^{2}}dwd\theta$ is also bounded by some constant which is independent of $u$.

By Lemma \ref{lm4.4}, we can   choose two positive constant $A_{3}$ and $A_{4}$ such that  the function
\begin{equation*}
  \varphi(t)=\left\{
               \begin{array}{ll}
                 \frac{1}{\sqrt{4\pi t}}+A_{3}t^{-\delta_{3}}, & \hbox{$0<t\le1$;} \\
                 A_{4}e^{-\delta_{4} t}, & \hbox{$t>1$,}
               \end{array}
             \right.
\end{equation*}
is continuous on $(0,\infty)$ and satisfies
\begin{align*}
\phi_{2}^{\ast\ast}(t)\leq &\frac{1}{t}\int_{0}^{t}\varphi(s)ds,\;\; 0<t\leq1;\\
\phi_{2}^{\ast}(t)\leq& \varphi(t),\;\;\;\;\;\;\;\;\;\;\;\;\;\;\;\;\;\;t>1.
\end{align*}
Then we can obtain
\begin{equation}\label{4.12}
 \phi_{2}^{\ast\ast}(t)\leq \frac{1}{t}\int_{0}^{t}\varphi(s)ds,\;\; t>0.
\end{equation}
In fact, for $t\geq 1$, we have
\begin{align*}
\phi_{2}^{\ast\ast}(t)=&\frac{1}{t}\left(\int_{0}^{1}\phi_{2}^{\ast}(s)ds+\int_{1}^{t}\phi_{2}^{\ast}(s)ds\right)\\
\leq&\frac{1}{t}\left(\phi_{2}^{\ast\ast}(1)+\int_{1}^{t}\varphi(s)ds\right)\\
\leq &\frac{1}{t}\left(\int_{0}^{1}\varphi(s)ds+\int_{1}^{t}\varphi(s)ds\right)\\
=&\frac{1}{t}\int_{0}^{t}\varphi(s)ds.
\end{align*}

Set $v=(-\partial_{w}^{2}-T_{a}-\lambda)^{1/2}u$. We have
\begin{align*}
|u(w,\theta)|=&|(-\partial_{w}^{2}-\Delta_{\mathbb{S}^{1}}-\lambda)^{-1/2}v|\\
\leq&\int_{\mathbb{R}\times \mathbb{S}^{1}}|v(w',\theta')|\phi_{2}(w-w',\theta-\theta')dw'd\theta'.
\end{align*}
We claim that
\begin{equation}\label{4.13}
 |u|^{\ast}(t)\leq|u|^{\ast\ast}(t)\leq  |v|^{\ast\ast}(t)\int_{0}^{t}\varphi(s)ds+\int_{t}^{\infty}|v|^{\ast}(s) \varphi(s)ds,\;\;t>0.
\end{equation}
In fact,  we have, by (\ref{2.10}) and (\ref{4.12}),
\begin{align*}
 |u|^{\ast\ast}(t)\leq& -\int_{t}^{\infty}s \phi_{2}^{\ast\ast}(s)d |v|^{\ast}(s)+\phi_{2}^{\ast\ast}(t)\int_{|v|^{\ast}(t)}^{\infty}m(|v|,s)ds\\
\leq&-\int_{t}^{\infty}\left(\int_{0}^{s}\varphi(r)dr\right)d |v|^{\ast}(s)+\frac{1}{t}\int_{0}^{t}\varphi(s)ds\int_{|v|^{\ast}(t)}^{\infty}m(|v|,s)ds\\
\leq& |v|^{\ast}(t)\int_{0}^{t}\varphi(s)ds+\int_{t}^{\infty}|v|^{\ast}(s) \varphi(s)ds+\frac{1}{t}\int_{0}^{t}\varphi(s)ds\int_{|v|^{\ast}(t)}^{\infty}m(|v|,s)ds\\
=&\frac{1}{t}\int_{0}^{t}\varphi(s)ds\left(t|v|^{\ast}(t)+\int_{|v|^{\ast}(t)}^{\infty}m(|v|,s)ds\right)+\int_{t}^{\infty}|v|^{\ast}(s) \varphi(s)ds\\
=&  |v|^{\ast\ast}(t)\int_{0}^{t}\varphi(s)ds+\int_{t}^{\infty}|v|^{\ast}(s) \varphi(s)ds.
\end{align*}
This proves the claim.

Therefore, using (\ref{4.13}) and closely following the proof of Theorem 1.7 in \cite{lly}, we have that for $0<\delta_{3}<\frac{1}{2}$, there exists a constant
C which is independent of $u$ and $\Omega(u)$ such that
\begin{equation*}
  \int_{\Omega(u)}e^{4\pi |u|^{2}}dwd\theta=\int_{0}^{|\Omega(u)|}e^{4\pi (|u|^{\ast}(t))^{2}}dt\leq
  \int_{0}^{1/\lambda}e^{4\pi (|u|^{\ast}(t))^{2}}dt<C.
\end{equation*}
The desired result follows.
\end{proof}
\
\

\textbf{Proof of Theorem \ref{th1.2}}. We firstly show that for $\epsilon$ small enough, there exists $\lambda'>0$ such that
\begin{equation}\label{4.14}
\begin{split}
&-\int_{\mathbb{R}\times \mathbb{S}^{1}}\overline{u} T_{a}u dwd\theta+\int_{\mathbb{R}\times \mathbb{S}^{1}}(|\partial_{w}u|^{2}+\lambda'|u|^{2})dwd\theta\\
\leq& \int_{\mathbb{R}\times \mathbb{S}^{1}}(|\partial_{w}u|^{2}+|(\partial_{\theta}-ia)u|^{2}+\lambda|u|^{2})dwd\theta,\;\; \;\;\forall u\in C_{0}^{\infty}(\mathbb{R}\times \mathbb{S}^{1}).
\end{split}
\end{equation}

Let $u(w,\theta)=\frac{1}{2\pi}\sum\limits_{n\in\mathbb{Z}}u_{n}(w)e^{in\theta}$, where $u_{n}(w)\in C_{0}^{\infty}(\mathbb{R})$. We compute
\begin{align*}
&\int_{\mathbb{R}\times \mathbb{S}^{1}}(|\partial_{w}u|^{2}+|(\partial_{\theta}-ia)u|^{2}+\lambda|u|^{2})dwd\theta+\int_{\mathbb{R}\times \mathbb{S}^{1}}T_{a}u \overline{u}dwd\theta-\int_{\mathbb{R}\times \mathbb{S}^{1}}(|\partial_{w}u|^{2}+\lambda'|u|^{2})dwd\theta\\
=&\sum_{n\in\mathbb{Z}}\left(|n-a|^{2}+\lambda-n^{2}+2a\frac{n^{2}}{\sqrt{n^{2}+\epsilon}}-\lambda'\right)\int_{\mathbb{R}}|u_{n}(w)|^{2}dw\\
=&\sum_{n\in\mathbb{Z}}\left(a^{2}+2a\frac{n^{2}}{\sqrt{n^{2}+\epsilon}}-2an+\lambda-\lambda'\right)\int_{\mathbb{R}}|u_{n}(w)|^{2}dw.
\end{align*}
If we choose   $0<\epsilon<\frac{1}{a}(\lambda+a^{2})$, then
\begin{align*}
\lambda+a^{2}+2a\min_{n\in\mathbb{Z}}\left(\frac{n^{2}}{\sqrt{n^{2}+\epsilon}}-n\right)
=&\lambda+a^{2}+2a\min_{n\geq1}\left(\frac{n^{2}}{\sqrt{n^{2}+\epsilon}}-n\right)\\
=&\lambda+a^{2}+2a\epsilon \min_{n\geq1}\frac{-n}{\sqrt{n^{2}+\epsilon}(n+\sqrt{n^{2}+\epsilon})}\\
>&\lambda+a^{2}+2a\epsilon \min_{n\geq1}\frac{-1}{2n}\\
=&\lambda+a^{2}-a\epsilon\\
 >&0.
\end{align*}
Therefore,  (\ref{4.14}) is valid for   $0<\lambda'<\lambda+a^{2}-a\epsilon$,  .

Substituting $w=\ln r$, we have, by (\ref{1.2}),
\begin{align*}
\int_{\mathbb{R}^{2}}|\nabla_{\mathbf{A}}u|^{2}dx+\lambda
\int_{\mathbb{R}^{2}}\frac{|u|^{2}}{|x|^{2}}dx=\int_{\mathbb{R}\times \mathbb{S}^{1}}(|\partial_{w}u|^{2}+|(\partial_{\theta}-ia)u|^{2}+\lambda|u|^{2})dwd\theta.
\end{align*}
By (\ref{4.14}) and Theorem \ref{th4.6}, we have that
there exists $C>0$ such that
\begin{align*}
\int_{\mathbb{R}^{2}}\frac{e^{4\pi|u|^{2}}-1}{|x|^{2}}dx\leq C
\end{align*}
for any complex-valued function $u\in C^{\infty}_0(\mathbb{R}^{2}\setminus\{0\})$ with
\begin{align*}
\int_{\mathbb{R}^{2}}|\nabla_{\mathbf{A}}u|^{2}dx+\lambda
\int_{\mathbb{R}^{2}}\frac{|u|^{2}}{|x|^{2}}dx\leq 1.
\end{align*}
The sharpness of the constant $4\pi$ can be verified by the process similar to that in
the proof of Theorem \ref{th3.1}. The proof of Theorem \ref{th1.2} ia thereby completed.\\
\

Finally, we shall prove Corollary \ref{co1.4}. The proof is similar to that given in \cite{ks}, Theorem 5.1.
\\
\

\textbf{Proof of Corollary \ref{co1.4}}.
Firstly, with the same argument as in \cite{ks},  we have
\begin{align}\label{4.16}
\liminf_{p\rightarrow\infty} p\mu_{p}(\lambda)\geq 8\pi e.
\end{align}

Nextly, we consider the test function
\begin{equation*}%\label{b3.3}
v_{\delta}(w,\theta)=\frac{1}{\sqrt{-2\pi\ln\delta+(\lambda+a^{2})\pi\left(\frac{1}{2}-\frac{1}{2}\delta^{2}+\delta^{2}\ln\delta\right)}}u_{\delta}(w,\theta),
\end{equation*}
where $u_{\delta}$ defined in (\ref{3.3}). Then
\begin{align*}
\int_{\mathbb{R}^{2}}|\nabla_{\mathbf{A}}v|^{2}dx+\lambda
\int_{\mathbb{R}^{2}}\frac{|v|^{2}}{|x|^{2}}dx=&\int_{\mathbb{R}\times \mathbb{S}^{1}}(|\partial_{w}v|^{2}+|(\partial_{\theta}-ia)v|^{2}+\lambda|v|^{2})dwd\theta\\
=&\int_{\mathbb{R}\times \mathbb{S}^{1}}(|\partial_{w}v|^{2}+|\partial_{\theta}v|^{2}+(\lambda+a^{2})|v|^{2})dwd\theta\\
=&1
\end{align*}
and
\begin{align*}
\int_{\mathbb{R}^{2}}\frac{|v|^{p}}{|x|^{2}}dx\geq&\int_{\{(w,\theta): w^{2}+\theta^{2}<\delta^{2}\}}|v|^{p}dwd\theta\\
=&\left(\frac{-\ln\delta}{\sqrt{-2\pi\ln\delta+(\lambda+a^{2})\pi\left(\frac{1}{2}-\frac{1}{2}\delta^{2}+\delta^{2}\ln\delta\right)}}\right)^{p}\pi \delta^{2}.
\end{align*}
Therefore,
\begin{align*}
\mu_{p}(\lambda)\leq& \frac{\int_{\mathbb{R}^{2}}|\nabla_{\mathbf{A}}v|^{2}dx+\lambda
\int_{\mathbb{R}^{2}}\frac{|v|^{2}}{|x|^{2}}dx}{(\int_{\mathbb{R}^{2}}\frac{|v|^{p}}{|x|^{2}}dx)^{2/p}}\\
\leq&\frac{-2\pi\ln\delta+(\lambda+a^{2})\pi\left(\frac{1}{2}-\frac{1}{2}\delta^{2}+\delta^{2}\ln\delta\right)}{\ln^{2}\delta}(\pi \delta^{2})^{-2/p}.
\end{align*}
Choosing $\delta=e^{-p/4}$, we obtain
\begin{align*}
p\mu_{p}(\lambda)
\leq& 4\frac{\pi p/2+(\lambda+a^{2})\pi\left(\frac{1}{2}-\frac{1}{2}e^{-p/2}-pe^{p/2}/4\right)}{-p/4}\pi ^{-2/p}e\rightarrow 8\pi e,\; p\rightarrow\infty.
\end{align*}
Therefore,
\begin{equation}\label{4.17}
\limsup\limits_{p\rightarrow\infty} p\mu_{p}(\lambda)\leq 8\pi e.
\end{equation}
The desired result follows by combining (\ref{4.16}) and (\ref{4.17}).

\section{Proof of Theorem \ref{th1.3}}
Notice that
\begin{equation}\label{b5.1}
  \begin{split}
&\int_{\mathbb{B}^{2}}|\nabla_{\mathbf{A}}u|^{2}dx+\lambda
\int_{\mathbb{B}^{2}}\frac{|u|^{2}}{|x|^{2}}dx-\frac{1}{4}\int_{\mathbb{B}^{2}}\frac{|u|^{2}}{|x|^{2}\ln^{2}|x|}dx\\
=&\int_{0}^{\infty}\int_{\mathbb{S}^{1}}\left(|\partial_{w}u|^{2}-\frac{1}{4}\frac{1}{w^{2}}|u|^{2}+|(\partial_{\theta}-ia)u|^{2}+\lambda|u|^{2}\right)dwd\theta.
  \end{split}
\end{equation}
We need only  to establish Truding-Moser inequality for $-\partial_{w}^{2}-\frac{1}{4w^{2}}-T_{a}+\lambda$.

Recall that the heat kernel for $-\partial_{w}^{2}-\frac{1}{4w^{2}}$ is given by (see \cite{ben})
\begin{align*}
e^{t(\partial_{w}^{2}+\frac{1}{4w^{2}})}=\frac{\sqrt{ww'}}{4\pi t}\int_{0}^{2\pi}e^{-\frac{w^{2}+w'^{2}-2ww'\cos\vartheta}{4t}}d\vartheta,\;\; w>0,\;w'>0.
\end{align*}
We have the following asymptotic estimates of the fractional power:
\begin{lemma}\label{lm5.1}
Let $\lambda>0$ and set $\phi_{4}=(-\partial_{w}^{2}-\frac{1}{4w^{2}}-\Delta_{\mathbb{S}^{1}}+\lambda)^{-\frac{1}{2}}$.
There holds, for some $\delta_{5}>0$,
\begin{align*}
\phi_{4}\leq& \frac{1}{2\pi} \frac{1}{\sqrt{(w-w')^{2}+(\theta-\theta')^2}}+O(1),\;\;|w-w'|\leq 1,\; |\theta-\theta'|\leq \pi;\\
\phi_{4}\lesssim& e^{-\delta_{5}|w-w'|},\;\;\;\;\;\;\;\;\;\;\;\;\;\;\;\;\;\;\;\;\;\;\;\;\;\;\;\;\;\;\;\;\;\;\;\;\;
\;\;\;\;|w-w'|\geq 1,\; |\theta-\theta'|\leq \pi.
\end{align*}
\end{lemma}
\begin{proof}
We have
\begin{align*}
\phi_{4}=&\frac{1}{\Gamma(1/2)}\int_{0}^{\infty}e^{-\lambda t}e^{t(\partial_{w}^{2}+\frac{1}{4w^{2}})}e^{t\Delta_{\mathbb{S}^{1}}}dt\\
=&\frac{1}{\Gamma(\frac{1}{2})(4\pi)^{3/2}}\sqrt{ww'}\int_{0}^{\infty}
\int_{0}^{\pi}t^{-2}e^{-\lambda t-\frac{w^{2}+w'^{2}-2ww'\cos\vartheta}{4t}}
 \sum_{n \in \Bbb{Z}} e^{- \frac{(\theta -\theta'- 2n \pi)^2}{4t}}dtd\vartheta\\
 =&\frac{1}{8\pi^{2}}\sqrt{ww'}\sum_{n \in \mathbb{Z}}\int_{0}^{\infty}\int_{0}^{2\pi}t^{-2}e^{-\lambda t-\frac{w^{2}+w'^{2}-2ww'\cos\vartheta+(\theta -\theta'- 2n \pi)^2}{4t}}
  dtd\vartheta.
\end{align*}
Substituting $t=\frac{1}{t}$ and using (\ref{2.3}), we obtain
\begin{equation}\label{c5.1}
  \begin{split}
\phi_{4}=&\frac{1}{8\pi^{2}}\sqrt{ww'}\sum_{n \in \mathbb{Z}}\int_{0}^{\infty}\int_{0}^{2\pi}e^{-\frac{\lambda}{t}-\frac{w^{2}+w'^{2}-2ww'\cos\vartheta+(\theta -\theta'- 2n \pi)^2}{4}t}
  dtd\vartheta\\
=&\frac{1}{8\pi^{2}}\sqrt{ww'} \sum_{n \in \mathbb{Z}}\int_{0}^{2\pi}2\left(\frac{4\lambda}{w^{2}+w'^{2}-2ww'\cos\vartheta+(\theta -\theta'- 2n \pi)^2}\right)^{\frac{1}{2}}\times\\
&K_{1}(\gamma^{1/2}\sqrt{w^{2}+w'^{2}-2ww'\cos\vartheta+(\theta-\theta'-2n\pi)^{2}})d\vartheta\\
=&\frac{\sqrt{\lambda}}{2\pi^{2}}\sqrt{ww'}\sum_{n \in \mathbb{Z}}\int_{0}^{2\pi}\frac{K_{1}(\gamma^{1/2}\sqrt{w^{2}+w'^{2}-2ww'\cos\vartheta+(\theta-\theta'-2n\pi)^{2}})}
{\sqrt{w^{2}+w'^{2}-2ww'\cos\vartheta+(\theta -\theta'- 2n \pi)^2}}d\vartheta\\
=&\frac{\sqrt{\lambda}}{\pi^{2}}\sqrt{ww'}\sum_{n \in \mathbb{Z}}\int_{0}^{\pi}\frac{K_{1}(\gamma^{1/2}\sqrt{w^{2}+w'^{2}-2ww'\cos\vartheta+(\theta-\theta'-2n\pi)^{2}})}
{\sqrt{w^{2}+w'^{2}-2ww'\cos\vartheta+(\theta -\theta'- 2n \pi)^2}}d\vartheta\\
=&: (I)+(II),
  \end{split}
\end{equation}
where
\begin{align*}
 (I)=&\sqrt{ww'}\int_{0}^{\pi}\frac{K_{1}(\gamma^{1/2}\sqrt{w^{2}+w'^{2}-2ww'\cos\vartheta+(\theta-\theta')^{2}})}
{\sqrt{w^{2}+w'^{2}-2ww'\cos\vartheta+(\theta-\theta')^{2}}}d\vartheta;\\
(II)=&\sqrt{ww'}\sum_{n \neq 0}\int_{0}^{\pi}\frac{K_{1}(\gamma^{1/2}\sqrt{w^{2}+w'^{2}-2ww'\cos\vartheta+(\theta-\theta'-2n\pi)^{2}})}
{\sqrt{w^{2}+w'^{2}-2ww'\cos\vartheta+(\theta -\theta'- 2n \pi)^2}}d\vartheta.
\end{align*}

We have, by (\ref{2.7})
\begin{align*}
(I)
\leq& \lambda^{-1/2}\sqrt{ww'}\int_{0}^{\pi}\frac{1}
{w^{2}+w'^{2}-2ww'\cos\vartheta+(\theta-\theta')^{2}}d\vartheta.
\end{align*}
Substituting $t=\frac{1+\cos\vartheta}{2}$ and using (\ref{2.1})-(\ref{2.2}), we get
\begin{equation}\label{5.1}
  \begin{split}
&\int_{0}^{\pi}\frac{\sqrt{ww'}}{w^{2}+w'^{2}-2ww'\cos\vartheta+(\theta-\theta')^{2}}d\vartheta\\
=&
\int_{0}^{1}\frac{\sqrt{ww'}}{w^{2}+w'^{2}-2ww'(2t-1)+(\theta-\theta')^{2}}(t(1-t))^{-1/2}dt\\
=& \frac{\sqrt{ww'}}{(w+w')^{2}+(\theta-\theta')^{2}}  \int_{0}^{1}\frac{1}{1-\frac{4ww'}{(w+w')^{2}+(\theta-\theta')^{2}}t}(t(1-t))^{-1/2}dt\\
=& \frac{\sqrt{ww'}}{(w+w')^{2}+(\theta-\theta')^{2}}\frac{\Gamma(\frac{1}{2})\Gamma(\frac{1}{2})}{\Gamma(1)}F\bigl(1,\frac{1}{2};1;\frac{4ww'}{(w+w')^{2}+(\theta-\theta')^{2}}\bigr)\\
=&\pi\frac{\sqrt{ww'}}{(w+w')^{2}+(\theta-\theta')^{2}}\left(1-\frac{4ww'}{(w+w')^{2}+(\theta-\theta')^{2}}\right)^{-\frac{1}{2}}
\times\\
&F\left(0,\frac{1}{2};1;\frac{4ww'}{(w+w')^{2}+(\theta-\theta')^{2}}\right)\\
=&\pi\frac{\sqrt{ww'}}{\sqrt{(w+w')^{2}+(\theta-\theta')^{2}}}\frac{1}{\sqrt{|w-w'|^{2}+(\theta-\theta')^{2}}}\\
\leq&\frac{\pi}{2}\frac{1}{\sqrt{|w-w'|^{2}+(\theta-\theta')^{2}}}.
  \end{split}
\end{equation}
To get the last inequality, we use the inequality $2\sqrt{ww'}\leq w+w'\leq \sqrt{(w+w')^{2}+(\theta-\theta')^{2}}$.
Therefore, we obtain
\begin{equation}\label{5.2}
  \begin{split}
(I)
\leq&\frac{1}{2\pi}\frac{1}{\sqrt{|w-w'|^{2}+(\theta-\theta')^{2}}},\;\; \;\;\;\;\;\;\;\;0<|\theta-\theta'|\leq \pi, \; w>0,\;w'>0.
  \end{split}
\end{equation}

On the other hand, if $|w-w'|>1$, then by (\ref{2.4}) and (\ref{5.1}),
\begin{equation}\label{5.3}
  \begin{split}
(I)
\lesssim & \sqrt{ww'}\int_{0}^{\pi}\frac{e^{-\gamma^{1/2}\sqrt{w^{2}+w'^{2}-2ww'\cos\vartheta+(\theta-\theta')^{2}}}}
{w^{2}+w'^{2}-2ww'\cos\vartheta+(\theta-\theta')^{2}}d\vartheta\\
\leq&e^{-\sqrt{\gamma((w-w')^{2}+(\theta-\theta')^{2})}}\int_{0}^{\pi}\frac{\sqrt{ww'}}
{w^{2}+w'^{2}-2ww'\cos\vartheta+(\theta-\theta')^{2}}d\vartheta\\
\leq& e^{-\sqrt{\gamma((w-w')^{2}+(\theta-\theta')^{2})}}\frac{1}{\sqrt{|w-w'|^{2}+(\theta-\theta')^{2}}}\\
\lesssim&e^{-\gamma^{1/2}|w-w'|}.
  \end{split}
\end{equation}

For $(II)$, we have,  $\forall w, w'>0$ and $|\theta-\theta'|\leq \pi$,
 \begin{align*}
 (II)
\lesssim& \sqrt{ww'}\sum_{n \neq 0}\int_{0}^{\pi}\frac{e^{-\gamma^{1/2}\sqrt{w^{2}+w'^{2}-2ww'\cos\vartheta+(\theta-\theta'-2n\pi)^{2}}}}
{w^{2}+w'^{2}-2ww'\cos\vartheta+(\theta -\theta'- 2n \pi)^2}d\vartheta\\
\leq&  \sum_{n \neq 0}e^{-\gamma^{1/2}\sqrt{|w-w'|^{2}+(\theta-\theta'-2n\pi)^{2}}}\int_{0}^{\pi}\frac{\sqrt{ww'}}
{w^{2}+w'^{2}-2ww'\cos\vartheta+(\theta -\theta'- 2n \pi)^2}d\vartheta\\
\leq&  \sum_{n \neq 0}e^{-\frac{\sqrt{\lambda}}{2}|w-w'|-\frac{\sqrt{\gamma}}{2}|\theta-\theta'-2n\pi|}\int_{0}^{\pi}\frac{\sqrt{ww'}}
{w^{2}+w'^{2}-2ww'\cos\vartheta+(\theta -\theta'- 2n \pi)^2}d\vartheta.
 \end{align*}
 To get the last equality, we use
 $$2\sqrt{|w-w'|^{2}+(\theta-\theta'-2n\pi)^{2}}\geq|w-w'|+|\theta-\theta'-2n\pi|.$$
Therefore, by using (\ref{5.1}), we have
\begin{equation}\label{5.4}
  \begin{split}
   (II)
\leq& e^{-\frac{\sqrt{\lambda}}{2}|w-w'|} \sum_{n \neq 0}e^{-\frac{\sqrt{\gamma}}{2}|\theta-\theta'-2n\pi|}\frac{1}{\sqrt{|w-w'|^{2}+(\theta-\theta'-2n\pi)^{2}}}\\
\lesssim& e^{-\frac{\sqrt{\lambda}}{2}|w-w'|},\;\;\;\;\;\;\;\;\;\;\; w>0,\;w'>0,\;\;|\theta-\theta'|\leq \pi.
  \end{split}
\end{equation}
The desired result follows by combing (\ref{c5.1}) and  (\ref{5.2})-(\ref{5.4}).

\end{proof}

With the same argument as in the proof of Lemma \ref{lm4.4}, we have the following lemma. Since the proof is similar, we omit it.
\begin{lemma}\label{lm5.3}
Let  $\delta_{6}>0$  and set $\phi_{5}=(-\partial_{w}^{2}-\frac{1}{4w^{2}}-T_{a}+\lambda)^{-\frac{1}{2}}$.
There holds, for some $\delta_{7}>0$,
\begin{align*}%\label{4.8}
|\phi_{5}|\leq& \frac{1}{2\pi} \frac{1}{\sqrt{(w-w')^{2}+(\theta-\theta')^2}}+O\left(\frac{1}{|w-w'|^{\delta_{6}}}\right),\;\;|w-w'|\leq 1,\; |\theta-\theta'|\leq \pi;\\
%\label{4.9}
|\phi_{5}|\leq& e^{-\delta_{7}|w-w'|},\;\;\;\;\;\;\;\;\;\;\;\;\;\;\;\;\;\;\;\;\;\;\;\;\;\;\;\;\;\;\;\;\;\;\;\;\;\;\;\;\;\;\;\;\;\;\;\;\;\;\;\;\;\;
\;\;\;\;|w-w'|\geq 1,\; |\theta-\theta'|\leq \pi.
\end{align*}
\end{lemma}

Notice that  $\phi_{2}$ and $\phi_{5}$ have the same asymptotic estimates. By following closely the proof of Theorem \ref{th4.6}, we have the following theorem. 
\begin{theorem}\label{th5.3}
Let $\lambda>0$ and $0\leq a\leq\frac{1}{2}$. There exists $C>0$ such that
\begin{align*}
\int_{0}^{\infty}\int_{\mathbb{S}^{1}}(e^{4\pi|u|^{2}}-1)dwd\theta\leq C
\end{align*}
for any complex-valued function $u\in C^{\infty}_0((0,\infty)\times \mathbb{S}^{1})$ with
\begin{align*}
-\int_{0}^{\infty}\int_{\mathbb{S}^{1}}T_{a}u \overline{u}dwd\theta+\int_{0}^{\infty}
\int_{\mathbb{S}^{1}}\left(|\partial_{w}u|^{2}-\frac{1}{4w^{2}}|u|^{2}+\lambda|u|^{2}\right)dwd\theta\leq 1.
\end{align*}
\end{theorem}
\
\
Now we can give the Proof of Theorem \ref{th1.3}. The proof is similar to that in the proof of Theorem \ref{th1.2}.
\\
\

\textbf{Proof of Theorem \ref{th1.3}}.
Let $\lambda'>0$ be such that for $\epsilon$ small enough,
\begin{equation*}%\label{4.14}
\begin{split}
&-\int_{0}^{\infty}\int_{\mathbb{S}^{1}}T_{a}u \overline{u}dwd\theta+\int_{0}^{\infty}\int_{\mathbb{S}^{1}}\left(|\partial_{w}u|^{2}-\frac{1}{4w^{2}}|u|^{2}+\lambda|u|^{2}\right)dwd\theta\\
\leq& \int_{0}^{\infty}\int_{\mathbb{S}^{1}}\left(|\partial_{w}u|^{2}-\frac{1}{4}\frac{1}{w^{2}}|u|^{2}+|(\partial_{\theta}-ia)u|^{2}+\lambda|u|^{2}\right)dwd\theta.
\end{split}
\end{equation*}
By Theorem \ref{th5.3} and (\ref{b5.1}), we have that there exists $C>0$ such that
\begin{align*}
\int_{0}^{\infty}\int_{\mathbb{S}^{1}}(e^{4\pi|u|^{2}}-1)dwd\theta\leq C
\end{align*}
for any complex-valued function $u\in C^{\infty}_0((0,\infty)\times \mathbb{S}^{1})$ with
\begin{align*}
&\int_{0}^{\infty}\int_{\mathbb{S}^{1}}\left(|\partial_{w}u|^{2}-\frac{1}{4}\frac{1}{w^{2}}|u|^{2}+|(\partial_{\theta}-ia)u|^{2}+\lambda|u|^{2}\right)dwd\theta\\
=&\int_{\mathbb{B}^{2}}|\nabla_{\mathbf{A}}u|^{2}dx+\lambda
\int_{\mathbb{B}^{2}}\frac{|u|^{2}}{|x|^{2}}dx-\frac{1}{4}\int_{\mathbb{B}^{2}}\frac{|u|^{2}}{|x|^{2}\ln^{2}|x|}dx\leq1.
\end{align*}
The sharpness of the constant $4\pi$ can be verified by the process similar to that in
the proof of Theorem \ref{th3.1}.
This completes the proof of Theorem \ref{th1.3}.


\begin{thebibliography}{99}

\bibitem{as}M. Abramowitz, I. A. Stegun,  Handbook of Mathematical Functions with Formulas, Graphs, and Mathematical Tables.
Dover, New York 1970.

 \bibitem{ad}D.R. Adams, A sharp inequality of J. Moser for higher order derivatives, Ann. of Math. 128 (2) (1988) 385-398.

\bibitem {Aer14}L. Aermark, Hardy and spectral inequalities for a class of
partial differential operators, PhD Thesis, Stockholm, 2014.

\bibitem {AFT03}B. Alziary, J. Fleckinger-Pell\'{e} and P. Tak\'{a}\v{c},
Eigenfunctions and Hardy inequalities for a magnetic Schr\"{o}dinger operator
in $\mathbb{R}^{2}$. Math. Methods Appl. Sci. 26 (2003), no. 13, 1093-1136.


\bibitem {BEL15}A. Balinsky, W.D. Evans and R.T. Lewis, The analysis and
geometry of Hardy's inequality. Universitext. Springer, Cham, 2015. xv+263 pp.


\bibitem {BLS04}A. Balinsky, A. Laptev and A. Sobolev, Generalized Hardy
inequality for the magnetic Dirichlet forms. J. Statist. Phys. 116 (2004), no.
1-4, 507-521.


\bibitem{ben}R. D. Benguria, R. L. Frank, M. Loss,  The sharp constant in the Hardy-Sobolev-Maz'ya inequality in
the three dimensional upper half space. Math. Res. Lett. 15, 613-622 (2008)



 \bibitem{bdel1} D. Bonheure,  J. Dolbeault, M. J. Esteban,  A. Laptev,
M. Loss, Symmetry results in two-dimensional inequalities
for Aharonov-Bohm magnetic fields, Comm. Math. Phys., 375 (2020),  2071-2087.

\bibitem {bdel2}D. Bonheure, J. Dolbeault, M. J. Esteban, A. Laptev and M.
Loss, Inequalities involving Aharonov-Bohm magnetic potentials in dimensions 2
and 3. Rev. Math. Phys. 33 (2021), no. 3, Paper No. 2150006, 29 pp.


\bibitem {CK16}C. Cazacu and D. Krej\v{c}i\v{r}\'{\i}k, The Hardy inequality
and the heat equation with magnetic field in any dimension, Comm. Partial
Differential Equations 41 (2016) no. 7, 1056-1088.


 \bibitem{CKLL} C. Cazacu, D. Krej\v{c}i\v{r}\'{\i}k, N. Lam and A. Laptev,
Hardy inequalities for magnetic p-Laplacians, arXiv:2201.02482.


\bibitem {ET99}D.E. Edmunds and H. Triebel, Sharp Sobolev embeddings and
related Hardy inequalities: the critical case. Math. Nachr. 207 (1999), 79-92.

\bibitem{er} A. Erd\'elyi, W. Magnus, F. Oberhettinger, F.G. Tricomi, Higher Transcendental Functions, vol. I,
Based, in part, on notes left by Harry Bateman, McGraw-Hill Book Company, Inc., New York-Toronto-London, 1953.



\bibitem{f}J. Faraut, Analysis on Lie groups. An introduction. Cambridge Studies in Advanced Mathematics,
110. Cambridge University Press, Cambridge, 2008.


\bibitem {FKLV20}L. Fanelli, D. Krej\v{c}i\v{r}\'{\i}k, A. Laptev and L. Vega,
On the improvement of the Hardy inequality due to singular magnetic fields.
Comm. Partial Differential Equations 45 (2020) 1202--1212.


 \bibitem{gr}I. S. Gradshteyn, L. M. Ryzhik, Table of Integrals, Series, and Products. 7th edition. Translation edited and with a preface by Alan Jeffrey and Daniel Zwillinger. Academic Press, Inc., San Diego, CA, 2007.






\bibitem{ks}D. Karmakar, K. Sandeep, Adams inequality on the Hyperbolic space, J. Func. Anal.  270(5)(2016), 1792-1817.



\bibitem {Kre19}D. Krej\v{c}i\v{r}\'{\i}k, Complex magnetic fields: an
improved Hardy-Laptev-Weidl inequality and quasi-self-adjointness. SIAM J.
Math. Anal. 51 (2019), no. 2, 790--807.


\bibitem{ll}N. Lam, G. Lu, Sharp Moser-Trudinger inequality in the Heisenberg
group at the critical case and applications, Adv. Math. 231 (6)
(2012), 3259-3287.

\bibitem{l} N. Lam, G. Lu, A new approach to sharp Moser-Trudinger and Adams type inequalities: a rearrangement-free argumnet,
    J. Diff. Equa. 255(2013), 298-325.

\bibitem{LamLu-VJM} N. Lam, G. Lu, Improved $L^{p}$-Hardy and $L^{p}$-Rellich inequalities with magnetic fields, to appear in Vietnam Journal of Mathematics, special issue in honor of Carlos Kenig's 70th birthday.



 \bibitem{lw}   A. Laptev,  T. Weidl, Hardy inequalities for magnetic Dirichlet forms, in Mathematical results in quantum mechanics (Prague, 1998), vol. 108 of Oper. Theory Adv. Appl.,
Birkh\"auser, Basel, 1999, pp. 299-305.


\bibitem{lly} J. Li, G. Lu,  Q. Yang, Fourier analysis and optimal Hardy-Adams  inequalities  on
 hyperbolic spaces of any even dimension, Adv. Math. 333 (2018), 350-385.

\bibitem{lly2} J. Li, G. Lu,  Q. Yang, Sharp Adams and Hardy-Adams inequalities of any fractional order on hyperbolic spaces of all dimensions, Trans. Amer. Math. Soc., 373(5)(2020), 3483-3513.

    \bibitem{LLWY}  X. Liang, G. Lu, X. Wang and Q. Yang,  Sharp Hardy-Trudinger-Moser inequalities in any N-dimensional hyperbolic spaces. Nonlinear Anal. 199 (2020), 112031, 19 pp.





\bibitem{ly2}G. Lu, Q. Yang, A  sharp Trudinger-Moser inequality on  any  bounded and convex   planar domain,  Calc. Var. Partial Differential Equations 55 (2016), no. 6, Art. 153, 16 pp.





 \bibitem{ly}G. Lu, Q. Yang, Sharp Hardy-Adams inequalities for bi-Laplacian on hyperbolic space of dimension four, 	
 Adv. Math. 319(2017),  567-598.



\bibitem{ly3}G. Lu, Q. Yang, Paneitz operators and Hardy-Sobolev-Maz'ya inequalities for higher order derivatives on half spaces, 	
Amer. J. Math., 141(2019), 1777-1816.

\bibitem{ly4}G. Lu, Q. Yang,  Green's functions of Paneitz and GJMS operators on hyperbolic spaces and sharp Hardy-Sobolev-Maz'ya inequalities on half spaces,
Adv. Math. 398(2022), 108156.

\bibitem{MWY-ANS} X. Ma, X. Wang and Q. Yang, Hardy-Adams inequalities on $H^2\times R^{n-2}$. Adv. Nonlinear Stud. 21 (2021), no. 2, 327-345.

\bibitem{mst}G. Mancini, K. Sandeep, K. Tintarev. Trudinger-moser inequality in the hyperbolic space $\mathbb{H}^{N}$,  Adv. Nonlinear Anal. 2(3)(2013), 309-324.




\bibitem{mo}J. Moser, A sharp form of an inequality by N. Trudinger, Indiana Univ. Math. J. 20 (1971), 1077-1092.


\bibitem{on}R. O'Neil, Convolution operateors and $L(p,q)$ sapces, Duke Math. J. 30(1963), 129-142.



\bibitem{Po}  S.I. Pohozaev, On the Sobolev embedding theorem for $pl=n$, in: Doklady Conference, Section
Math., Moscow Power Inst., 1965, pp. 158-170.



\bibitem {Sol94}M. Solomyak, A remark on the Hardy inequalities. Integral
Equations Operator Theory 19 (1994), no. 1, 120-124.



\bibitem{sw} E.M. Stein, G. Weiss, Introduction to Fourier Analysis in Euclidean Spaces, Princeton Univ. Press, Princeton, NJ, 1971.

\bibitem{ta} M. E. Taylor, Partial Differential
Equations II, Qualitative Studies of Linear Equations, Second Edition, Springer Science+Business Media, LLC 1996, 2011.


\bibitem{t}N.S. Trudinger, On imbeddings into Orlicz spaces and some applications, J. Math. Mech. 17 (1967), 473-483.


\bibitem{wy}G. Wang, D. Ye, A Hardy-Moser-Trudinger inequality, Adv. Math. 230 (2012), 294-320.



\bibitem{Yu}  V.I. Yudovic,
Some estimates connected with integral operators and with solutions of elliptic equations,
Dokl. Akad. Nauk SSSR 138 (1961) 805-808 (in Russian).





\end{thebibliography}
\end{document}